\documentclass{article}

\usepackage[utf8]{inputenc}
\usepackage[T1]{fontenc}
\usepackage[english]{babel}
\usepackage{amsmath}
\usepackage{amsthm}
\usepackage{amssymb}
\usepackage{amsfonts}
\usepackage{dsfont}
\usepackage{graphicx}
\usepackage[lofdepth,lotdepth]{subfig}
\usepackage{mathtools,mathrsfs} %nice calligraphy
\usepackage{cases} %for accolades with numbered equations
\usepackage{tikz}

\newtheorem{theorem}{Theorem}[section]

\newtheorem{lemma}[theorem]{Lemma}
\newtheorem{corollary}[theorem]{Corollary}

\newtheorem*{proposition*}{Proposition}
\newtheorem*{acknowledgement}{Acknowledgement}
\theoremstyle{remark}
\newtheorem{remark}[theorem]{Remark}

\numberwithin{equation}{section}

\newcommand{\R}{\mathbb{R}}

\newcommand{\rn}{\mathbb{R}^n}

\newcommand\abs[1]{\left|#1\right|}
\newcommand\map[5]{\begin{array}{ccccc}
#1 & : & #2 & \to & #3 \\
& & #4 & \mapsto & #5 \\
\end{array}}

\title{Laplacian of the distance function on the cut locus on a riemannian manifold}
\author{François Générau}
\date{December 2, 2019}
\begin{document}

\maketitle

\begin{abstract}
  In this note, we prove that, on a riemannian manifold $M$, the laplacian of the distance function to a point $b$ is $-\infty$ in the sense of barriers, at every point of the cut locus of $M$ with respect to $b$. We apply this result to an obstacle-type variational problem with the distance function as an obstacle, in order to replace it with a smoother obstacle.
\end{abstract}
{\bf Keywords:} distance function, cut locus, obstacle problem, Riemannian manifold.

\section{Introduction}

Let $M$ be a Riemannian manifold of dimension $n \geq 2$ without boundary, and $b$ a point in $M$. We denote by $d$ the distance function on $M$ and $d_b$ the distance function to the point $b$. Let $T>0$ and $\gamma : [0,T] \to M$ be a unit speed geodesic such that $\gamma(0) = b$, $t_0\in (0,T)$ and $p = \gamma(t_0)$.
We say that $p$ is a \emph{cut point} of $b$ along $\gamma$ if $\gamma$ is length minimizing between $b$ and $p$, but not after $p$, \textit{i.e} $d_b(\gamma(t)) = t$ for $t\leq t_0$, and $d_b(\gamma(t)) < t$ for $t>t_0$.
The \emph{cut locus} of $M$ with respect to $b$, denoted by $Cut_b(M)$, is defined as the set of all cut points of $b$. Some well known facts related to the cut locus include:
\begin{itemize}
  \item $Cut_b(M)$ is the set of points $p$ in $M$, such that either there exist at least two minimizing geodesics from $b$ to $p$, or the points $b$ and $p$ are \emph{conjugate} (see for instance \cite[proposition 4.1]{sakai1996riemannian}),
  \item $Cut_b(M)$ is the closure of the set of points $p$ in $M$, such that there exist at least two minimizing geodesics from $b$ to $p$,
  \item the function $d_b$ is smooth outside $Cut_b(m)$ (see \cite[proposition 4.8]{sakai1996riemannian}),
  \item given $p\in M$, if there are at least two minimizing geodesics from $b$ to $p$, then the function $d_b$ is not differentiable at $p$ (see \cite[proposition 4.8]{sakai1996riemannian}).
\end{itemize}

Our main result is the following theorem. It describes the non-smooth behavior of the distance function at each point of the cut locus, in terms of its laplacian, in the sense of barriers.

\begin{theorem}\label{lemma:laplacian on the cut locus}
  Let $M$ be a smooth manifold without boundary of dimension ${n\geq 2}$, and $b\in M$. Let $d_b$ be the distance function to the point $b$, and $Cut_b(M)$ the cut locus with respect to $b$. The laplacian of $d_b$ on $Cut_b(M)$ is $-\infty$ in the sense of barriers. More precisely, for any $p\in Cut_b(M)$ and $A>0$, there exists a smooth function $\phi$ defined on a neighborhood of $p$ such that
  \begin{equation*}
    \phi\geq d_b, \quad \phi(p)=d_b(p) \quad \text{and} \quad {\Delta \phi(p)\leq -A}.
  \end{equation*}
\end{theorem}

\begin{remark}
  One could also prove an analog result for the cut locus with respect to a submanifold of $M$ instead of a point $b$. We restricted ourselves to this case for simplicity.
\end{remark}

This result proves useful when one wants to apply the maximum principle to a function involving the distance function $d_b$. As an application, we use it to prove the corollary \ref{cor:cor} below.

Our motivation is the following obstacle-type variational problem, and its relation with $Cut_b(M)$:
\begin{equation}\label{eq:variational problem}
  \min_{\substack{u\in H^1(M)\\ u\leq d_b}} \int_{M} \abs{\nabla u}^2 -m\int_{M}u,
\end{equation}
where $m\in(0,+\infty)$.
This problem can be seen as an extension to manifolds of the elastic-plastic torsion problem, which was extensively studied in the 60's and 70's - see for instance \cite{caffarelli_friedman_elastoplastic} (Cafarelli and Friedman), and the references therein. Indeed, the original elastic-plastic torsion problem boils down to the following problem:
\begin{equation}\label{eq:elastic-plastic problem}
  \min_{\substack{u\in H^1_0(\Omega)\\ u\leq d_{\partial \Omega}}} \int_{\Omega} \abs{\nabla u}^2 -m\int_{\Omega}u,
\end{equation}
where $\Omega$ is a bounded lipschitz open subset of $\rn$, and $d_{\partial \Omega}$ the distance function to its boundary. The analog of the cut locus in this flat setting is called the \emph{ridge} of $\Omega$ in \cite{caffarelli_friedman_elastoplastic}. It is the closure of the set of points of $\Omega$ that have at least two closest points on the boundary $\partial \Omega$.
Let $u_m$ and $v_m$ be the minimizers of \eqref{eq:variational problem} and \eqref{eq:elastic-plastic problem} respectively.
In \cite{caffarelli_friedman_elastoplastic}, using the fact that $\abs{\nabla v_m}\leq 1$, the authors show that the ridge of $\Omega$ is contained in the set $\{v_m<d_{\partial \Omega}\}$.
In problem \eqref{eq:variational problem}, the inequality $\abs{\nabla u_m}\leq 1$ doesn't hold in general. However, using the theorem \ref{lemma:laplacian on the cut locus} above, we can still prove that
\begin{equation}
  Cut_b(M) \subset \{u_m<d_b\}. \label{eq:inclusion}
\end{equation}
Indeed, theorem \ref{lemma:laplacian on the cut locus} yields the following:
\begin{corollary}\label{cor:cor}
  For any $m>0$, there exists a function $\widetilde{d_b}$ that is smooth on $M\setminus\{b\}$, such that $u_m \leq \widetilde{d_b} \leq d_b$, $\widetilde{d_b}=d_b$ in a neighborhood of $b$ and $\widetilde{d_b}<d_b$ on $Cut_b(M)$. In particular, the solution $u_m$ of \eqref{eq:variational problem} is also the solution of the following variational problem with a smooth obstacle:
  \begin{equation}
    \label{eq:variational problem with regularized d}
    \min_{\substack{u\in H^1(M)\\ u\leq \widetilde{d_b}}} \int_{M} \abs{\nabla u}^2 -m\int_{M}u.
  \end{equation}
\end{corollary}

In turn, this corollary implies \eqref{eq:inclusion}. What is more, it can be used to show that, with the regularity theory for obstacle-type problems, $u_m$ is locally $C^{1,1}$ regular. For a broader study of problem \eqref{eq:variational problem} and its relation with the cut locus, we refer to \cite{generau_oudet_velichkov}.

Note that in \cite{mantegazza2014distributional} (and in particular theorem 2.10), C.Mantegazza, G.Mascellani and G.Uraltsev gave a complete description of the distributional hessian and laplacian of $d_b$ using geometric measure theory tools. Here we are interested in inequalities in the sense of barriers, and we use mainly tools from riemannian geometry.

This paper is organized as follows. In section \ref{section:proof of the proposition}, we give a proof of theorem \ref{lemma:laplacian on the cut locus}.
In section \ref{section:application to a variational problem}, we apply the theorem to prove corollary \ref{cor:cor}.

\begin{acknowledgement}
  We would like to thank Bozhidar Velichkov for some very useful discussions and ideas used in this paper.
\end{acknowledgement}

\section{Proof of theorem \ref{lemma:laplacian on the cut locus}}\label{section:proof of the proposition}
\begin{proof}
Let $p\in Cut_b(M)$ and $A>0$. We know that either there exist two minimizing geodesics from $b$ to $p$, or there exists a unique minimizing geodesic from $b$ to $p$ along which the two points are conjugate.

\emph{Case one}. There exist two minimizing geodesics from $b$ to $p$. This implies that $d_b$ is not differentiable at $p$. Let $\delta>0$ be smaller than the distance between $p$ and $Cut_p(M)$, and $\psi:B(p,\delta)\to \rn$ be a normal coordinate map around $p$.
According to \cite[Proposition 3.4]{mantegazza_mennucci_2003}, there exists $C>0$ such that the function $x\mapsto C \abs{x}^2-d_b\circ\psi^{-1}(x)$ is convex. As it is not differentiable at $0$, it has at least two subgradients $v$ and $w$. We have
\begin{equation}
  \label{eq:two tangent planes} C\abs{x}^2-d_b\circ\psi^{-1}(x) \geq \max(v\cdot x,w \cdot x) - d_b(p),
\end{equation}
with equality for $x=0$. For $B>0$ to be chosen big enough later, let us define the function $f:\rn \to \R$ by
\[f(x) := \frac{1}{2}(v\cdot x + w \cdot x) + B(v\cdot x - w \cdot x)^2 -  d_b(p).\]
Then for $x$ in a neighborhood of $0$ we have
\begin{align}
  f(x) &= \max(v\cdot x,w \cdot x) - \frac{1}{2}\abs{v\cdot x - w \cdot x} + B(v\cdot x - w \cdot x)^2 - d_b(p) \nonumber \\
      &\leq \max(v\cdot x,w \cdot x)  -  d_b(p)
  \label{eq:two planes estimate}
\end{align}
with equality at $x=0$. Setting $\phi=C\abs{\psi}^2-f\circ\psi$, we get from \eqref{eq:two tangent planes} and \eqref{eq:two planes estimate} that for $q$ in a neighborhood of $p$,
\begin{equation*}
  d_b(q)\leq \phi(q),
\end{equation*}
with equality at $q=p$. What is more, as $\psi$ is a normal coordinate map, we have
\begin{equation}
   \Delta \phi(p)
   = \Delta (\phi \circ \psi^{-1})(0)
   = 2nC-2B\abs{v-w}^2. \nonumber
\end{equation}
In particular if $B$ is big enough, then we have $\Delta \phi(p)\leq -A$, which concludes case one.

\emph{Case two}. There exists a unique minimizing geodesic $\gamma$ such that $\gamma(0)=b$ and $\gamma(1)=p$. For $r>0$ to be chosen small enough later, let $q$ be the intersection point of $\gamma$ with the sphere $\partial B(p,r)$, and $R = d_b(q)$. For $\epsilon>0$, let us define
\begin{align}
  p_\epsilon&:=\gamma(1+\epsilon), \nonumber\\
  r_\epsilon&:=r+d(p,p_\epsilon). \nonumber
\end{align}
\begin{figure}
 \begin{center}
%   \begin{tikzpicture}
%     \draw (0,0) node {$\bullet$}; \draw (0,0) node[below] {$b$};
%     \draw (5,0) node {$\bullet$}; \draw (5,0) node[below] {$p$};
%     \draw (5.5,0) node {$\bullet$}; \draw (5.5,0) node[below] {$p_\epsilon$};
%     \draw (4,0) node {$\bullet$}; \draw (4,0) node[below left] {$q$};
%     \draw [domain=pi/2:3*pi/2] plot ({5+cos(\x r)},{0+sin(\x r)});
%     \draw [domain=pi/2:3*pi/2] plot ({5.5+1.5*cos(\x r)},{0+1.5*sin(\x r)});
%     \draw [domain=2*pi/3:4*pi/3] plot ({7+3*cos(\x r)-0.5*(7+3*cos(\x r)-4)^2},{0+3*sin(\x r)});
%     \draw (5.5,1.5) node[right] {$\partial B(p_\epsilon,r_\epsilon)$};
%     \draw (5,1) node[right] {$\partial B(p,r)$};
%     \draw (4.3,2.5) node[left] {$\partial B(b,R)$};
%     \draw (5.5,-1.5) node[right] {$g_\epsilon$};
%     \draw (5,-1) node[right] {$g$};
%     \draw (4.3,-2.5) node[left] {$f$};
%     \draw[->] (0,0) -- (2,0);\draw[-] (2,0) -- (5.5,0);
%     \draw (2,0) node[above] {$\gamma$};
%
%
%   \end{tikzpicture}
\includegraphics[height=6cm]{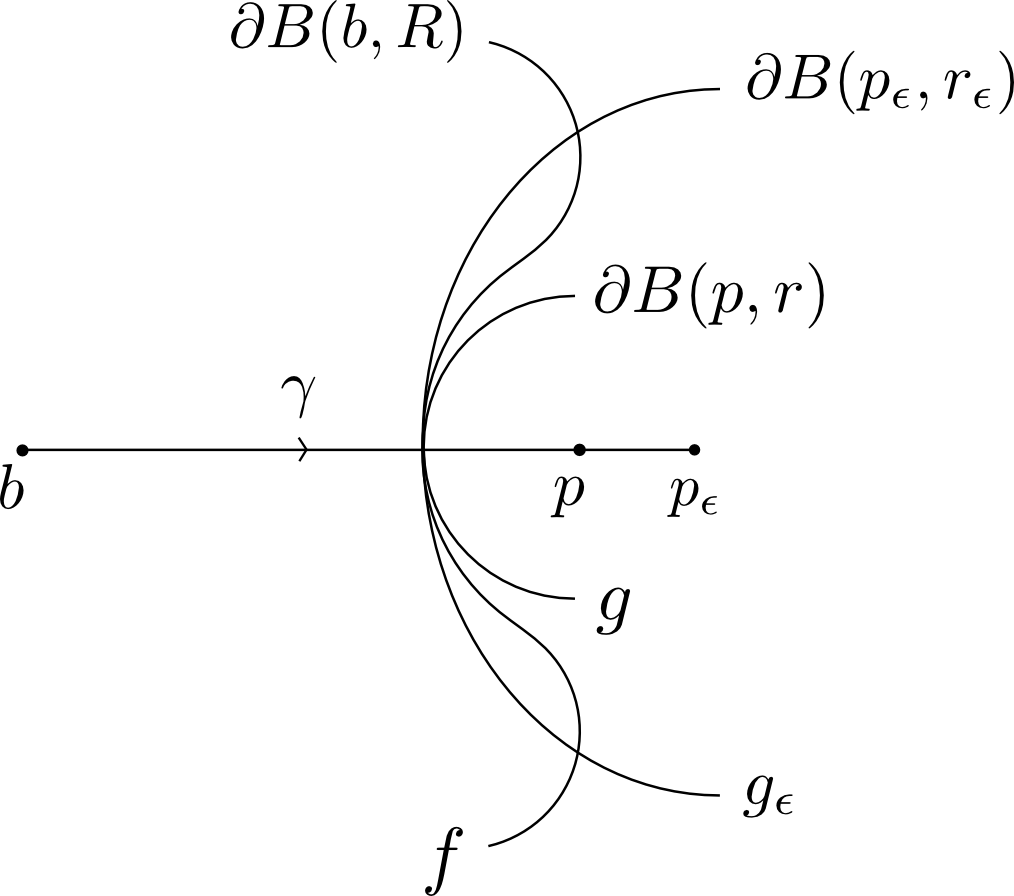}
 \end{center}

\caption{Construction of  $r$, $R$, $q$, $p_\epsilon$ , $r_\epsilon$, $f$, $g$, $g_\epsilon$.}
\label{fig:conjugate point}
\end{figure}
We may choose $r$ and $\epsilon$ so that $p$ and $p_\epsilon$ are contained in a totally normal geodesic ball $B(q,\delta)$ for some $\delta>0$.
This implies in particular that for any points $p'$ and $q'$ in neighborhoods of $p$ and $q$ respectively, there exists a unique minimizing geodesic $\mu_{(p',q')}$ between $p'$ and $q'$, and $\mu_{(p',q')}$ depends continuously on $(p',q')$ (see \cite[Remark 3.8]{carmo1992riemannian}). Let us fix some normal coordinates at $q$ on $B(q,\delta)$ such that the unit normal vector to $\partial B(b,R)$ at $q$ is $\partial_n$. We will use these normal coordinates throughout the rest of the proof.
In these coordinates, around $q$, $\partial B(b,R)$ is the graph of a smooth function $f$ such that $f(0)=0$ and $\nabla f(0)=0$.
Likewise, $\partial B(p,r)$ and $\partial B(p_\epsilon,r_\epsilon)$ can be seen around $q$ as the graphs of some smooth functions $g$ and $g_\epsilon$ respectively, such that $g(0)=g_\epsilon(0)=0$ and $\nabla g(0)=\nabla g_\epsilon(0)=0$.

\noindent\emph{Step one.} We prove that
\begin{equation}
  \exists v\in \R^{n-1},\;v \neq 0, \quad v \cdot Hf(0) v = v \cdot Hg(0) v, \label{eq:same curvature bis}
\end{equation}
where $HF$ denotes the hessian matrix of the function $F$.
As the points $b$ and $p$ are conjugate, there exists a Jacobi field $J$ along $\gamma$ such that $J(0)=0$, $J(1)=0$ and $J(t)\neq 0 $ for $t\in(0,1)$.
We may extend it to a global smooth field $\widetilde{J}$ on $M$. Let $(\Phi_{\widetilde{J}}^s)_{s\in\R}$ be the flow of this vector field, defined by the following equation:
\begin{equation*}
  \forall x \in M, \; \forall s \in \R, \; \frac{\mathrm{d}}{\mathrm{d}s}\Phi_{\widetilde{J}}^s(x) = \widetilde{J}(\Phi_{\widetilde{J}}^s(x)).
\end{equation*}
We define a variation $\Gamma$ of the curve $\gamma$ as follows:
\[\forall s\in \R, \forall t\in[0,1],\;\Gamma(s,t) := \Phi_{\widetilde{J}}^s(\gamma(t)).\]
For any $s\in \R$, $\Gamma(s,\cdot)$ is still a curve from $b$ to $p$, as $\widetilde{J}(0)=0$ and $\widetilde{J}(1)=0$. As $\gamma$ is a geodesic, we have
\[\frac{\mathrm{d}}{\mathrm{d}s}[\mathrm{length}(\Gamma(s,\cdot))]_{s=0} = 0.\]
As $\Gamma$ is a fixed point variation of $\gamma$ whose variation field $J$ is a Jacobi field, we also have (see \cite[Proposition 10.14]{lee2006riemannian} for instance).
\begin{equation}
  \frac{\mathrm{d^2}}{\mathrm{d}s^2}[\mathrm{length}(\Gamma(s,\cdot))]_{s=0} = 0. \label{eq:length second variation}
\end{equation}
By the implicit function theorem, for any $s$ close enough to $0$, there exist some unique times $t_R(s)$ and $t_r(s)$ such that
\[d(\Gamma(s,t_R(s)),b) = R \quad \text{and} \quad d(\Gamma(s,t_r(s)),p) = r.\]
Moreover, the functions $s\mapsto t_R(s)$ and $s\mapsto t_r(s)$ are smooth. Now we set
\[M(s):=\Gamma(s,t_R(s)) \quad \text{and} \quad N(s) := \Gamma(s,t_r(s)).\]
We have
\begin{align}
  \mathrm{length}(\Gamma(s,\cdot))
  &= \mathrm{length}({\Gamma(s,\cdot)}_{|_{[0,t_R(s)]}}) + \mathrm{length}({\Gamma(s,\cdot)}_{|_{[t_R(s),t_r(s)]}}) \nonumber \\ &\qquad +\mathrm{length}({\Gamma(s,\cdot)}_{|_{[t_r(s),1]}}) \nonumber \\
  & \geq d(\Gamma(s,0),\Gamma(s,t_R(s))) + d(\Gamma(s,t_R(s)),\Gamma(s,t_r(s))) \nonumber \\
  & \qquad + d(\Gamma(s,t_r(s)),\Gamma(s,1))  \nonumber \\
  & = R + d(M(s),N(s)) + r. \nonumber
\end{align}
Meanwhile, as $s$ goes to $0$, thanks to \eqref{eq:length second variation}, we also have
\begin{align*}
  \mathrm{length}(\Gamma(s,\cdot)) = \mathrm{length}(\gamma) + o(s^2) = R+r+o(s^2),
\end{align*}
so $d(M(s),N(s))\leq o(s^2)$. In our normal coordinates around $q$, we have $d(x,y)\geq c \abs{x-y}$ where $c>0$ is a constant and $\abs{\:\cdot\:}$ is the euclidean norm. This comes from the fact that the metric $g$ of $M$ and the euclidean metric are locally equivalent in any coordinates system. With this remark, we get
\begin{equation}
  M(s) = N(s) + o(s^2), \label{eq:equality order 2}
\end{equation}
where the equality is to be understood in coordinates. For any $s$ small, let $x_M(s),x_N(s)\in \R^{n-1}$ be such that $M(s) = (x_M(s),f(x_M(s)))$ and $N(s) = (x_N(s),g(x_N(s)))$. The functions $x_M$ and $x_N$ are smooth because $M$ and $N$ are. Using the fact that $\nabla f(0)=\nabla g(0)=0$, \eqref{eq:equality order 2} gives
\begin{align}
  x_M'(0) \cdot Hf(0) x_M'(0) = x_N'(0) \cdot Hg(0) x_N'(0). \label{eq:same curvature}
\end{align}
Let us show that $x_N'(0)=x_M'(0)$. We have
\begin{align}
  M'(0) &= \partial_s \Gamma(0,t_R(0)) + \partial_t \Gamma(0,t_R(0))t_R'(0) \nonumber \\
        & = J(t_R(0)) + \dot{\gamma}(t_R(0))t_R'(0). \label{eq:M'(0)}
\end{align}
By definition the curve $M$ is on the sphere $\partial B(b,R)$, so $M'(0)$ is tangent to this sphere and orthogonal to $\dot{\gamma}(t_R(0))$ by the Gauss lemma. Also, as $J(0)=0$ and $J(1)=0$, the Jacobi field $J$ is normal to $\gamma$ (see \cite[lemma 10.6]{lee2006riemannian}), so $J(t_R(0))$ is orthogonal to $\dot{\gamma}(t_R(0))$.
Combined with these facts, \eqref{eq:M'(0)} yields $M'(0) = J(t_R(0))$. Likewise, we have $N'(0) = J(t_r(0)) = J(t_R(0))$. So $M'(0) = N'(0)$, and consequently $x_M'(0) = x_N'(0)$.
What is more, as $J(t)\neq 0$ for $t\in(0,1)$, we have $x_M'(0)\neq 0$. So setting $v = x_M'(0)$, with \eqref{eq:same curvature} we get \eqref{eq:same curvature bis}.

\noindent\emph{Step two.} Now we want to show that
\begin{equation}
  v \cdot Hg(0) v > v \cdot Hg_\epsilon(0) v. \label{eq:different curvature}
\end{equation}
Let us first show that $g\geq g_\epsilon$ in a neighborhood of $0$. Let us argue by contradiction and assume that for some $x\in \R^{n-1}$, we have $g(x) < g_\epsilon(x)$. Let $\mu :[0,1] \to M$ be the shortest geodesic from the point of coordinate $(x,g(x))$ to the point of coordinate $(0,r)$, \textit{i.e.} $p$. If $x$ has been taken close enough to $0$, then for $t\in[0,1]$, $\mu(t)$ stays inside the normal neighborhood of $q$, $B(q,\delta)$, on which our normal coordinates are defined.
For $t\in[0,1]$, let $x(t)\in\R^{n-1}$ and $z(t)\in \R$ be such that in coordinates, $\mu(t) = (x(t),z(t))$. We have $z(0)=g(x)<g_\epsilon(x) = g_\epsilon(x(0))$. And $z(1) = r > 0 = g_\epsilon(0) = g_\epsilon(x(1))$. As $\mu$ is continuous, there exists $t\in(0,1)$ such that $z(t) = g_\epsilon(x(t))$, \textit{i.e.} $\mu(t)\in \partial B(p_\epsilon,r_\epsilon)$. This implies
\begin{align*}
  r_\epsilon
  &= d(\mu(t),p_\epsilon) \\
  &\leq d(\mu(t),p) + d(p,p_\epsilon) \\
  &< d(\mu(0),p) + d(p,p_\epsilon) \\
  &= r + d(p,p_\epsilon) \\
  &= r_\epsilon,
\end{align*}
which gives a contradiction. We conclude that $g\geq g_\epsilon$. In particular, for any $w\in \R^{n-1}$, we have $w \cdot Hg(0) w \geq w \cdot Hg_\epsilon(0) w$. Thus, the matrix $Hg(0)-Hg_\epsilon(0)$ is symetric non-negative. To show that it is positive definite and conclude that \eqref{eq:different curvature} holds, we only need to show that
\begin{equation}
  \forall w\in \R^{n-1}, \quad Hg(0) w \neq Hg_\epsilon(0) w. \label{eq:different hessian}
\end{equation}
To prove this, let us define the exponential map $\exp_S$ associated with a smooth oriented hypersurface $S$ of $M$.
Let $\mathrm{Exp}$ be the global exponential map, defined on the whole tangent space of $M$. Let $\nu_S$ be the unit normal vector to $S$. We set
\begin{equation}
  \nonumber \map{\exp_{S}}{S \times [0,\infty]}{M}{(\theta,t)}{Exp(t\nu_S(\theta))}
\end{equation}
The geodesic $\gamma$ is minimizing between $p_\epsilon$ and $\partial B(p_\epsilon,r_\epsilon)$. In particular, the point $p$ is not a cut point of $q$ along $\gamma$, so lemma \ref{lemma:smoothness} below shows that the differential of the map $\exp_{\partial B(p_\epsilon,r_\epsilon)}$ is invertible at $(q,r)$.
On the contrary, the map $\exp_{\partial B(p,r)}(\cdot,r)$ is constant and so its differential is null at $q$. In particular, for any $w\in T_q\partial B(p,r)$, we have
 \begin{align}\label{eq:normal different differential}
   \nabla_w \nu_{\partial B(p_\epsilon,r_\epsilon)}(q) \neq \nabla_w \nu_{\partial B(p,r)}(q),
 \end{align}
 where $\nabla_w$ denotes the covariant derivative in the direction $w$.
 Let us define the families of vectors $(u_i^g)_{1\leq i \leq n-1}$ and $(u_i^{g_\epsilon})_{1\leq i \leq n-1}$ at the point of coordinates $(x,z)\in \R^{n-1}\times \R$, by $u_i^g(x,z) := \partial_i + \partial_i g(x) \partial_n$ and $u_i^{g_\epsilon}(x,z) := \partial_i + \partial_i g_\epsilon(x) \partial_n$.
 These vectors form a basis of the tangent spaces of $\partial B(p,r)$ and $\partial B(p_\epsilon,r_\epsilon)$ respectively, as we recall that these surfaces are the graphs of the functions $g$ and $g_\epsilon$.
 Furthermore, the two basis are identical at $(0,0)$. Thus, \eqref{eq:normal different differential} is equivalent to
 \begin{align*}
   \exists i\in\{1,..,n-1\}, \; u_i^{g_\epsilon} \cdot \nabla_w \nu_{\partial B(p_\epsilon,r_\epsilon)}(q) \neq u_i^g \cdot \nabla_w \nu_{\partial B(p,r)}(q).
 \end{align*}
 Note that if some vector fields $X$ and $Y$ are orthogonal, then we have $X\cdot Y = 0$, and so $\nabla_w X \cdot Y + X \cdot \nabla_w Y = 0$. So we get
 \begin{align*}
   \exists i\in\{1,..,n-1\}, \nabla_w u_i^{g_\epsilon} \cdot  \nu_{\partial B(p_\epsilon,r_\epsilon)}(q) &\neq \nabla_w u_i^g \cdot \nu_{\partial B(p,r)}(q), \\
   \textit{i.e.} \quad \exists i\in\{1,..,n-1\}, \nabla_w u_i^{g_\epsilon} \cdot \partial_n &\neq \nabla_w u_i^g \cdot \partial_n.
 \end{align*}
 Recalling that in normal coordinates, the tangent vectors $(\partial_j)_{1\leq j \leq n}$ have vanishing covariant derivatives at $0$, we get
 \begin{align*}
   \exists i\in\{1,..,n-1\}, w^j\partial_{ji}g_\epsilon(0) &\neq w^j\partial_{ji}g(0),
 \end{align*}
 which is \eqref{eq:different hessian}. So \eqref{eq:different curvature} holds.

 \noindent\emph{Step three.} Putting steps one and two together, we get
 \begin{equation}
   \exists v \neq 0, \quad v \cdot (Hg_\epsilon(0) - Hf(0))v < 0. \nonumber
 \end{equation}
 In particular, there exists a basis $(v_i)_{1\leq i \leq n-1}$ of $\R^{n-1}$ with $v_1 = v$, that is orthogonal for the quadratic form $q_\epsilon(w) = w \cdot (Hg_\epsilon(0) - Hf(0))w$. Let $k_\epsilon$ be a quadratic form of $\R^{n-1}$ such that the $(v_i)$ are orthogonal for $k_\epsilon$, $k_\epsilon(v) = 0$ and for $i\geq2$, $k_\epsilon(v_i) = \max(0,q_\epsilon(v_i)+1)$.
 This way we have $k_\epsilon>q_\epsilon$, and $k_\epsilon\geq0$. Now let us set ${h_\epsilon} = g_\epsilon - k_\epsilon$. Note that in dimension $2$, we actually have $h_\epsilon = g_\epsilon$, and the rest of the proof is a bit less technical.
 By construction we have
 \begin{align}
   {h_\epsilon} &\leq g_\epsilon, \label{eq:leq h}\\
   {h_\epsilon} &= g_\epsilon \quad \text{on $\R v$}, \label{eq:= h}\\
   {h_\epsilon} &\leq f \quad \text{on $B(0,\rho_\epsilon)$, for some $\rho_\epsilon>0$,}\label{eq:leq f}
 \end{align}
 where \eqref{eq:leq f} comes from the fact that the hessian of ${h_\epsilon}$ at $0$ verifies $H{h_\epsilon}(0) = Hg_\epsilon(0) - k_\epsilon < Hg_\epsilon(0) - q_\epsilon = Hf(0)$, and ${h_\epsilon}(0) = f(0) = 0$, $\nabla{h_\epsilon}(0) = \nabla f(0) = 0$.
 It will be convenient, at the end of the proof, to have a smooth function $h: \R^{n-1} \to \R$ such that
 \begin{equation}\label{eq:barrier h}
   h(0) = 0,\quad \nabla h(0) = 0, \quad \text{and for all $\epsilon>0$ sufficiently small},\quad h_\epsilon \geq h.
 \end{equation}
 To see that such function exists, we note that for $\epsilon<\epsilon'$, we have $g_\epsilon\geq g_{\epsilon'}$. This can be proven the same way we proved $g_\epsilon\leq g$ at the beginning of step two. We then fix $\epsilon'>0$.
 Therefore, for any $\epsilon$ sufficiently small, we have $g_{\epsilon'}\leq g_\epsilon \leq g$. This allows us to see that the quadratic forms $q_\epsilon$, and then $k_\epsilon$, are bounded independently of $\epsilon$, by a quadratic form $k$. In turn, this implies that the function $h = g_{\epsilon'} - k$ verifies \eqref{eq:barrier h}.
 Let $S_{h_\epsilon}$ be the hypersurface of $M$ defined in our normal coordinates around $q$ as the graph of the function ${h_\epsilon}$ on $B(0,\rho_\epsilon)$.
 Let us set
 \begin{equation*}
   \phi_\epsilon = d(\cdot, S_{h_\epsilon}) + R.
 \end{equation*}
 We will show that for $\epsilon$ sufficiently small, $\phi_\epsilon$ is the function $\phi$ we are looking for to prove the theorem.

 Let us first show that $\phi_\epsilon \geq d_b$ in a neighborhood of $p$. Let $p'\in B(q,\delta)$ be a point near $p$. Let $\mu:[0,1] \to M$ be a length minimizing geodesic between $S_{h_\epsilon}$ and $p'$.
 For $t\in [0,1]$, let $(x(t),z(t))\in \R^{n-1}\times \R$ be the coordinates of $\mu(t)$. As $\mu(0)\in S_{h_\epsilon}$, we have $z(0) = {h_\epsilon}(x(0)) \leq f(x(0))$. Moreover, the coordinates $(x,z) = (0,r)$ of $p$ verify $z>f(x)$.
 Therefore, provided $p'$ is close enough to $p$, we have $z(1) \geq f(x(1))$. Thus, there exists $t\in[0,1]$ such that $z(t) = f(x(t))$, \textit{i.e.} $\mu(t)\in\partial B(b,R)$. This implies
 \begin{equation}
   d(\cdot, S_{h_\epsilon}) \geq d(\cdot, \partial B(b,R)) \quad \text{on a neighborhood of $p$}, \label{eq:S under B}
 \end{equation}
 and so $\phi_\epsilon \geq d_b$ on a neighborhood of $p$.

 As $q\in S_{h_\epsilon}$, we also have $\phi_\epsilon(p) \leq d(p,q) + R = d_b(p)$. So we get $\phi_\epsilon(p) = d_b(p)$ as well.

 We are left to show that $\phi_\epsilon$ is smooth around $p$, and that given $A >0$, we have $\Delta \phi_\epsilon (p) \leq -A$, provided $\epsilon$ has been taken small enough. Let us first show that $\phi_\epsilon$ is smooth around $p$. Between $p_\epsilon$ and $q$, the geodesic $\gamma$ is minimizing among all geodesics from $p_\epsilon$ to $\partial B(p_\epsilon,r_\epsilon)$.
 The same way we have shown that ${h_\epsilon}\leq f$ implies $d(\cdot, S_{h_\epsilon}) \geq d(\cdot, \partial B(b,R))$ , one can show that ${h_\epsilon}\leq g_\epsilon$ implies
 \begin{equation}
 d(\cdot, S_{h_\epsilon}) \geq d(\cdot, \partial B(p_\epsilon,r_\epsilon))\quad \text{on a neighborhood of $p$}. \label{eq:S under B epsilon}
 \end{equation}
 Therefore, between $q$ and $p_\epsilon$, $\gamma$ is still minimizing among geodesics from $p_\epsilon$ to $S_{h_\epsilon}$. We may apply lemma \ref{lemma:smoothness} to conclude that $\phi_\epsilon$ is smooth at $p$.
 Now for $t$ in a neighborhood of $0$, let us define $c(t) := \exp_{\partial B(p_\epsilon,r_\epsilon)}(r,\exp_q(tv,g_\epsilon(tv)))$. We have,
 \begin{align*}
   d(c(t), S_{h_\epsilon})
   &\geq d(c(t), \partial B(p_\epsilon,r_\epsilon)) \quad \text{because of \eqref{eq:S under B epsilon},}\\
   &= d(c(t),\exp_q(tv,g_\epsilon(tv))) \quad \text{because of lemma \ref{lemma:smoothness},}\\
   &= d(c(t),\exp_q(tv,{h_\epsilon}(tv))) \quad \text{because of \eqref{eq:= h},} \\
   &\geq d(c(t), S_{h_\epsilon}).
 \end{align*}
 Therefore, the first inequality is actually an equality:
 \begin{equation*}
   d(c(t), S_{h_\epsilon})
   = d(c(t), \partial B(p_\epsilon,r_\epsilon)).
 \end{equation*}
 As $d(\cdot, S_{h_\epsilon})$ and $d(\cdot, \partial B(p_\epsilon,r_\epsilon))$ have the same gradient at $c(0)=p$, we deduce that
 \begin{equation*}
   \nabla^2_{\dot{c}(0),\dot{c}(0)}[d(\cdot, S_{h_\epsilon})]_{p}
   = \nabla^2_{\dot{c}(0),\dot{c}(0)}[d(\cdot, \partial B(p_\epsilon,r_\epsilon)]_{p}.
 \end{equation*}
 We have $\dot{c}(0) \neq 0$ because of lemma \ref{lemma:smoothness} again. Setting $w_1 := \dot{c}(0)/\abs{\dot{c}(0)}$, we get
 \begin{equation}\label{eq:second derivative w_1}
   \nabla^2_{w_1,w_1}[d(\cdot, S_{h_\epsilon})]_{p}
   = \nabla^2_{w_1,w_1}[d(\cdot, \partial B(p_\epsilon,r_\epsilon)]_{p}.
 \end{equation}
 Let us complete $w_1$ into an orthonormal basis $(w_1,..,w_n)$ of $TpM$.
 Let $S_{h}$ be the hypersurface of $M$ defined in our normal coordinates at $q$ as the graph of the function $h$ from \eqref{eq:barrier h}. The same way we have proved \eqref{eq:S under B}, one can prove that $d(\cdot, S_{h_\epsilon}) \leq d(\cdot, S_{h})$. As these two functions are equal up to order $1$ at $p$, we deduce that
 \begin{equation*}
   \nabla^2_{w_i,w_i}[d(\cdot, S_{h_\epsilon})]_{p}
   \leq \nabla^2_{w_i,w_i}[d(\cdot, S_{h}]_{p}.
 \end{equation*}
 Moreover, we know from \cite[proposition 3.4]{mantegazza_mennucci_2003} that there exists a constant $C>0$ such that for all $i$, $\nabla^2_{w_i,w_i}[d(\cdot, S_{h}]_{p} \leq C$, so we get
 \begin{equation*}
   \forall i \geq 2, \nabla^2_{w_i,w_i}[d(\cdot, S_{h_\epsilon})]_{p} \leq C,
 \end{equation*}
 where $C$ is independent of $\epsilon$. Combining this inequality with \eqref{eq:second derivative w_1}, we get
 \begin{equation}\label{eq:laplacian of phi}
   \Delta \phi_\epsilon(p) \leq (n-1)C + \nabla^2_{w_1,w_1}[d(\cdot, \partial B(p_\epsilon,r_\epsilon)]_{p}.
 \end{equation}
 Let us pick some normal coordinates $(x^i)_{1\leq i \leq n}$ at $p_\epsilon$, such that we have $\partial_n = \dot{\gamma}$ on the curve $\gamma$.
 By the Gauss lemma, the set of points that are equidistant to $\partial B(p_\epsilon,r_\epsilon)$ are orthogonal to the geodesics starting out orthogonally from $\partial B(p_\epsilon,r_\epsilon)$. Therefore we have $\dot{c}(0) \cdot \dot{\gamma}(1) = 0$, or equivalently $w_1 \cdot \partial_n = 0$.
 In particular, if $(w_1^i)$ are the coordinates of $w$ relatively to the normal coordinates $(x^i)$ at $p_\epsilon$, we have $w_1^n = 0$.
 In these coordinates we have $d(x,\partial B(p_\epsilon,r_\epsilon)) = r_\epsilon - \abs{x}$, so
 \begin{align*}
   \nabla^2_{w_1,w_1}[d(\cdot, \partial B(p_\epsilon,r_\epsilon)](x)
   &= - \nabla^2_{w_1,w_1} \abs{x} \\
   &= -w_1^iw_1^j\left(\partial_{ij}\abs{x} - \Gamma_{ij}^k \partial_k \abs{x}\right) \\
   &=-w_1^iw_1^j\left(-\frac{\delta_{ij}}{\abs{x}} + \frac{x_ix_j}{\abs{x}^3} + \Gamma_{ij}^k(x) \frac{x_k}{\abs{x}} \right),
 \end{align*}
 where the $(\Gamma^k_{ij})$ are the Christoffel symbols.
 We apply this formula at the point $p$ of coordinates $(0,..,0,-d(p,p_\epsilon))$, and obtain
 \begin{align}\label{eq:second derivative in coordinates}
   \nabla^2_{w_1,w_1}[d(\cdot, \partial B(p_\epsilon,r_\epsilon)](p)
    = \frac{w_1^iw_1^j \delta_{ij}}{d(p,p_\epsilon)} - w_1^iw_1^j\Gamma^n_{ij}(p).
 \end{align}
 As we are in normal coordinates, there exist a continuous function $\rho :[0,+\infty)\to [0,+\infty)$ such that $\rho(0)=0$ and
 \begin{align*}
   \abs{\Gamma^k_{ij}(x)} \leq \rho(\abs{x}) \quad \text{and} \quad \abs{g^{ij}(x) - \delta_{ij}} \leq \abs{x}\rho(\abs{x}).
 \end{align*}
 Moreover, note that we only consider coordinates centered at some points $p_\epsilon$ contained in a bounded neighborhood of $p$, so the function $\rho$ is independent of $\epsilon$. With these remarks, \eqref{eq:second derivative in coordinates} yields
 \begin{align*}
   \abs{\nabla^2_{w_1,w_1}[d(\cdot, \partial B(p_\epsilon,r_\epsilon)](p)
   - \frac{\abs{w}^2}{d(p,p_\epsilon)}} \leq 2\left(\sum\limits_{i}w_1^i\right)^2\rho(d(p,p_\epsilon)).
 \end{align*}
 As $\epsilon\to 0$, the right hand side goes to $0$, and $\frac{\abs{w}^2}{d(p,p_\epsilon)} = \frac{1}{d((p,p_\epsilon))} \to +\infty$, so
 \begin{equation*}
   \nabla^2_{w_1,w_1}[d(\cdot, \partial B(p_\epsilon,r_\epsilon)](p) \to -\infty,
 \end{equation*}
 which concludes the proof with \eqref{eq:laplacian of phi}.
\end{proof}
The following lemma is a well known fact when the hypersurface $S$ considered is replaced by a point. We give a proof here, as we couldn't find it in the literature.
\begin{lemma}\label{lemma:smoothness}
  Let $S$ be a smooth hypersurface of $M$ such that $\overline{S}$ is compact. Let $\mu:[0,1]\to M$ be a geodesic such that $q:=\mu(0)\in S$. We assume that $\mu$ is length minimizing between $p :=\mu(1)$ and $\overline{S}$, and that $p$ is not a cut point of $q$ along $\mu$. Then the exponential map $\exp_S$ is a diffeomorphism from a neighborhood $U(q)$ of $q$ to a neighborhood $U(p)$ of $p$, and on $U(p)$ we have
  \begin{equation*}
    d(\cdot, S) = \pi_2 \circ (\exp_{S_{|U(q)}})^{-1},
  \end{equation*}
  where $\pi_2:S\times \R \to \R$ is the projection on the second coordinate. In particular, $d(\cdot, S)$ is smooth around $p$.
\end{lemma}
\begin{proof}
  As $p$ is not a cut point of $q$ along $\gamma$, it is not a focal point of $S$ (see \cite[lemma 2.11]{sakai1996riemannian}). In particular, $\exp_S$ is a local diffeomorphism from $(q,r)$ to $p$.
  As $q$ is the unique closest point to $p$ on the compact set $\overline{S}$, for any neighborhood $U(q)$ of $q$, there exists a neighborhood $U(p)$ of $p$ such that for any point $p' \in U(p)$ the closest points to $p'$ on $\overline{S}$ are inside $U(q)$. We choose such $U(p)$ and $U(q)$ so that $\exp_S$ is a diffeomorphism from $U(q)$ to $U(p)$. Let $p'\in U(p)$ and $r:=d(p',S)$.
  A closest point to $p'$ on $S$ is a point $q'\in U(q)$ such that $p' = \exp_S(q',r)$, and so $d(p',S) = r = \pi_2 \circ (\exp_{S_{|U(q)}})^{-1}(p')$. This concludes the proof.
\end{proof}

\section{Application to a variational problem}\label{section:application to a variational problem}
Recall that $u_m$ is a minimizer in \eqref{eq:variational problem}. We give a proof of corollary \ref{cor:cor}.

\begin{proof}[Proof of corollary \ref{cor:cor}]
  It is a consequence of the minimality of $u_m$ in \eqref{eq:variational problem} that we have in the sense of distributions:
  \begin{equation}\nonumber
    \Delta u_m \geq -m.
  \end{equation}
  This implies, from the theory of subharmonic distributions, that $u_m$ has an upper semi-continuous representative. It is proved in \cite[section 9]{harvey_lawson_subharmonic}, in the present setting of manifolds. For a reference about subharmonic functions in the euclidean space, see for instance \cite{demailly2007complex}.

  Let $x\in Cut_b(M)$. According to theorem \ref{lemma:laplacian on the cut locus} we can find a smooth function $\phi$ defined on an open neighborhood $N$ of $x$ such that:
  \begin{equation}\nonumber
    \phi\geq d_b, \quad \phi(x)=d_b(x) \quad \text{and} \quad \Delta \phi \leq -m-1.
  \end{equation}
  Then the function $u_m-\phi$ is non-positive strictly subharmonic on $N$ and so by the maximum principle for subharmonic functions, we have
  \[(u_m-\phi)(x)<0, \quad\textit{i.e.} \quad u_m(x)< d_b(x).\]

  Thus, the function $u_m-d_b$ is upper semi-continuous and negative on the compact set $Cut_b(M)$, so there exists $\epsilon>0$ such that the set $\{u_m < d_b - \epsilon\}$ is an open set containing $Cut_b(M)$. Now choose any smooth bump function $\rho:M\to[0,1]$ that is compactly supported in $\{u_m < d_b - \epsilon\}\setminus \{b\}$ and such that $\rho=1$ on $Cut_b(M)$, and a smooth function $d_{b,\epsilon}:M\to \R$ such that
  \begin{equation}
    \label{eq:regularization} \|d_{b,\epsilon}-(d_b-\frac{\epsilon}{2})\|_{L^\infty}< \frac{\epsilon}{2}.
  \end{equation}
  We set
  \begin{equation}
    \label{eq:dtilde}
    \widetilde{d_b}:=(1-\rho)d_b+\rho d_{b,\epsilon}.
  \end{equation}
  It is now an easy exercise to show that $\widetilde{d_b}$ verifies the desired conditions.
  % This function is clearly smooth on $M\setminus{b}$. On the complement of the support of $\rho$, we have $d_{b,\epsilon}=d_b$, so $u_m \leq \widetilde{d_b} = d_b$. On the support of $\rho$, \eqref{eq:dtilde} and \eqref{eq:regularization} imply
  % \[\widetilde{d_b}\leq (1-\rho)d_b+\rho((d_b-\frac{\epsilon}{2})+\frac{\epsilon}{2})\leq d_b, \]
  % and
  % \begin{align*}
  %   \widetilde{d_b} &\geq (1-\rho)d_b + \rho((d_b-\frac{\epsilon}{2})-\frac{\epsilon}{2})\\
  %   &\geq (1-\rho)(d_b-\epsilon) + \rho(d_b-\epsilon) \\
  %   &= d_b-\epsilon\\
  %   &\geq u_m,
  % \end{align*}
  % so $u_m \leq \widetilde{d_b} \leq d_b$. Therefore the inequality $u_m \leq \widetilde{d_b} \leq d_b$ is verified on the whole manifold $M$. As already mentioned, we also have $\widetilde{d_b} = d_b$ on the complement of the support of $\rho$, which is a neighborhood of $b$. At last, we have $\widetilde{d_b} = d_{b,\epsilon} < d_b$ on $Cut_b(M)$. This concludes the proof.
\end{proof}
% \begin{proposition}
% The function $u_m$ has an upper semi-continuous representative and
% \begin{equation}
%   Cut_b(M) \subset \{u_m < d_b \}.
% \end{equation}
% \end{proposition}
% \begin{remark}
%   It can actually be shown, using a theorem of Brézis and Stampacchia on the regularity of elliptic inequalities, that the function $u_m$ is $C^1$.
% \end{remark}
% \begin{proof}
% It is a consequence of the minimality of $u_m$ in \eqref{eq:variational problem} that we have in the sense of distributions:
% \begin{equation}\nonumber
%   \Delta u_m \geq -m.
% \end{equation}
% This implies, from the theory of subharmonic distributions, that $u_m$ has an upper semi-continuous representative. It is proved in \cite[section 9]{harvey_lawson_subharmonic}, in the present setting of manifolds. For a reference about subharmonic functions in the euclidean space, see for instance \cite{demailly2007complex}.
%
% Let $p\in Cut_b(M)$. According to proposition \ref{lemma:laplacian on the cut locus}, we can find a smooth function $\phi$ defined on an open neighborhood $N$ of $p$ such that:
% \begin{equation}\nonumber
%   \phi\geq d_b, \quad \phi(p)=d_b(p) \quad \text{and} \quad \Delta \phi \leq -m-1.
% \end{equation}
% Then, the function $u_m-\phi$ is non-positive strictly subharmonic on $N$, and so by the maximum principle for subharmonic functions, we have
% \[(u_m-\phi)(p)<0, \quad\textit{i.e.} \quad u_m(p)< d_b(p).\]
% \end{proof}

%
\bibliographystyle{plain}
\bibliography{../../biblio/biblio}

\bigskip\noindent
François Générau:
Laboratoire Jean Kuntzmann (LJK),
Universit\'e Joseph Fourier\\
Bâtiment IMAG, 700 avenue centrale,
38041 Grenoble Cedex 9 - FRANCE\\
{\tt francois.generau@univ-grenoble-alpes.fr}\\

\end{document}